\documentclass[12pt,a4paper]{amsart}
\usepackage[utf8]{inputenc}
\usepackage[T1]{fontenc}
\usepackage{amsmath}
\usepackage{amssymb}
\usepackage{amsfonts}
\usepackage{amsthm}
\usepackage[all]{xy} 
\usepackage{tikz-cd}
\usepackage{hyperref}
\usepackage[shortlabels]{enumitem}
\usepackage{multicol}
\usepackage[enableskew]{youngtab}
\usepackage{ytableau}
\usepackage{multicol}
\usepackage{bbold}
\usepackage{geometry}
\usepackage{geometry}
\usepackage{apptools}
\usepackage{xcolor}

\theoremstyle{plain}
\newtheorem*{tw*}{Theorem}
\newtheorem{tw}{Theorem}[section]
\newtheorem*{lemma*}{Lemma}
\newtheorem{lemma}[tw]{Lemma}
\newtheorem{pro}[tw]{Proposition}
\newtheorem*{pro*}{Proposition}
\newtheorem{cor}[tw]{Corollary}
\newtheorem*{cor*}{Corollary}

\theoremstyle{definition}
\newtheorem{df}[tw]{Definition}
\newtheorem*{df*}{Definition}

\theoremstyle{remark}
\newtheorem{rem}[tw]{Remark}
\newtheorem*{rem*}{Remark}
\newtheorem{ex}[tw]{Example}
\newtheorem*{ex*}{Example}

\newcommand{\Hilb}[1]{\operatorname{Hilb}_{{#1}}}

\newcommand{\Perm}{\operatorname{S}}

\newcommand{\vv}{v}
\newcommand{\ch}{\operatorname{ch}}

\newcommand{\id}{\operatorname{id}}
\newcommand{\pt}{\operatorname{pt}}
\newcommand{\res}{\operatorname{res}}
\newcommand{\val}{\operatorname{val}}
\newcommand{\ZZ}{\mathbb{Z}}
\newcommand{\NN}{\mathbb{N}}
\newcommand{\QQ}{\mathbb{Q}}

\newcommand{\CC}{\mathbb{C}}
\newcommand{\Aa}{\mathbb{A}}
\newcommand{\TT}{\mathbb{T}}
\newcommand{\Tq}{{\mathbb{T}_q}}

\def\O{\mathcal{O}}
\def\L{\mathcal{L}}

\newcommand{\E}{\mathcal{E}}
\newcommand\znak[1]{(-1)^{l({#1})-1}}
\newcommand{\Partq}[1]{[{#1}]_q}
\newcommand{\Fr}{\operatorname{Fr}}

\newcommand{\xtto}[1]{\stackrel{#1}{\longrightarrow}}
\newcommand{\xto}[1]{{\xrightarrow{#1}}}
\newcommand{\mono}{\hookrightarrow}

\newcommand{\coh}{\operatorname{H}}
\DeclareMathOperator{\KTh}{K}

\title[The multiplicative structure of the McKay correspondence]{The multiplicative structure of the K-theoretical McKay correspondence for the Hilbert scheme of points in the complex plane}
\author{Jakub Koncki}
\address{Institute of Mathematics of the Polish Academy of Sciences, Poland}
\email{j.koncki@mimuw.edu.pl}

\author{Magdalena Zielenkiewicz}
\address{Institute of Mathematics, University of Warsaw, Poland}
\email{magdaz@mimuw.edu.pl}

\begin{document}

\begin{abstract}
We consider the K-theory of the Hilbert scheme of points in the complex plane, which under McKay correspondence is isomorphic to the space of symmetric functions $\Lambda^n$. 
We prove a formula conjectured by Boissière for the endomorphism of $\Lambda^n$ induced by multiplication by the classes of the Adams powers of the tautological bundle.
 We describe the structure constants for the multiplication on $\Lambda^n$ induced by the tensor product in K-theory.
\end{abstract} 

\maketitle
\section{Introduction}
\noindent
Hilbert schemes, introduced in the 60's \cite{Groth}, quickly became standard objects in algebraic geometry. In this paper, we focus on the Hilbert scheme of $n$ points in the complex plane $\Hilb{n}:=\Hilb{n}(\Aa^2_\CC)$. It is a smooth quasiprojective variety parametrizing zero-dimensional subschemes of $\Aa^2_\CC$ of length $n$. The McKay correspondence forms a bridge between the geometry of $\Hilb{n}$ and the combinatorics of symmetric functions of \hbox{degree $n$.}
\\ \\
The permutation group $\Perm_n$ of $n$ elements acts on the vector space $(\CC^{2})^n$ by permuting the $\CC^2$ factors. This action preserves the standard symplectic form on $\CC^{2n}$. The quotient space is a symplectic singularity in the sense of  \cite{Kal,Bea} and the Hilbert scheme $\Hilb{n}$ is its symplectic resolution. The McKay correspondence predicts a relation between the geometry of $\Hilb{n}$ and the $\Perm_n$-equivariant geometry of $\CC^{2n}$.
Various instances of the McKay correspondence are known. It was studied as an isomorphism of cohomology or K-theory groups e.g. \cite{Kal2,ItoNakajima}, as an equality of elliptic classes \cite{BoLi,Waelder} or as an equivalence of derived categories \cite{BKR}. The K-theoretical correspondence was used by Haiman to prove several combinatorial conjectures \cite{Haiman,Haiman2}. The study of the McKay correspondence (often for different groups and resolutions of singularities) is an active area of research, e.g. \cite{Ohmoto,Krug,AW2,Craw,AW}. 
\\ \\
We focus on the K-theoretical correspondence for the Hilbert scheme $\Hilb{n}(\Aa^2_\CC)$. It provides us with an isomorphism
$$\Theta: \Lambda^n \simeq \KTh_{\Perm_n}(\CC^{2n}) \to \KTh(\Hilb{n}(\Aa^2_\CC))\,,$$
where $\Lambda^n$ is the space of symmetric functions of degree $n$. The codomain of this isomorphism has a product structure given by the tensor product. It induces interesting operations on the space $\Lambda^n$. We are interested in two natural questions:
\begin{enumerate}[A)]
	\item \label{intro1} For a given element $F\in\KTh(\Hilb{n})$, find a compact formula for the map \hbox{$\E_F:\Lambda^n\to\Lambda^n$} such that 
	$\Theta(\E_F(x))=\Theta(x)\otimes F\,,$
	 i.e. the diagram below commutes
	 $$
	 \begin{tikzcd}
	 	{\Lambda^n} &  {\KTh(\Hilb{n})} \\
	 	 {\Lambda^n} & {\KTh(\Hilb{n})}
	 	\arrow["\Theta", from=1-1, to=1-2]
	 	\arrow["{-\otimes F}"', from=1-2, to=2-2]
	 	\arrow["{\E_F}"', from=1-1, to=2-1]
	 	\arrow["\Theta", from=2-1, to=2-2]
	 \end{tikzcd}
	$$
	\item Describe the product structure on $\Lambda^n$ induced from $\KTh(\Hilb{n})$. \label{intro2}
\end{enumerate}
The Hilbert scheme $\Hilb{n}$ comes with a distinguished vector bundle $B_n$ called the tautological vector bundle. Kirwan surjectivity type results imply that the cohomology $\coh^*(\Hilb{n})$ is generated as a ring by the graded parts of the Chern character $\ch^k(B_n)$ \cite{ES} and the K-theory is generated by the Adams powers $\psi^kB_n$ \cite{QuiverKirwan}. 
Therefore, 
it is enough to solve problem $\ref{intro1}$ for $F=\psi^kB_n$. Characteristic classes of the tautological bundles and their relation with the McKay correspondence can also be considered for varieties other than the affine plane, e.g. \cite{BoiNiep,Scala,Krug2}.
\\ \\
A cohomological version of problem $\ref{intro1}$ was solved in \cite{Lehn}. It was used in \cite{LS} to describe the multiplicative structure on $\Lambda^n$ induced from $\coh^*(\Hilb{n})$.
In K-theory, the formula for the map $\E_{B_n}$ was given in \cite{Boissiere}. Under the standard isomorphism $\bigoplus_{n\in\NN}\Lambda_n \simeq \QQ[p_1,p_2,\dots]$ (see Section \ref{s:sym}) we have:
$$
\E_{B_n}=\operatorname{Coeff}\left(t^0;\sum_{r\ge 1}rp_rt^r\cdot \exp\left(\sum_{r\ge 1}\frac{\partial}{\partial p_r}t^{-r}\right)\right)\,.
$$
Moreover, the formula for $F=\psi^k{B_n}$ was conjectured.
\begin{tw*} [{\cite[Conjecture 8.1]{Boissiere}}]
	$$\E_{\psi^k{B_n}}=k^{-1}\cdot\operatorname{Coeff}\left(t^0;\sum_{r\ge 1}rp_rt^r\cdot \exp\left(k\cdot\sum_{r\ge 1}\frac{\partial}{\partial p_r}t^{-r}\right)\right)\,.$$
\end{tw*}
Our main result is Theorem \ref{tw:main1}, where we prove an equivalent version of the above formula.  Both the results of \cite{Boissiere} and our proof rely on the study of the torus action.
\\

The standard action of the two dimensional complex torus $\TT=(\CC^*)^2$ on the affine plane induces an action on the Hilbert scheme $\Hilb{n}$. This action was first used to compute the Betti numbers of the Hilbert scheme in \cite{ESBBdec} and its use is a standard method since then, e.g. \cite{Evain2,Smirnov,KZ,SV}. 
The fixed point set $\Hilb{n}^\TT$ is finite, which allows for a combinatorial approach to the equivariant cohomology and K-theory using localization theorems \cite{Tho,AB}. Roughly speaking, any class in K-theory is determined by local contributions coming from the fixed point set. These contributions are governed by the combinatorics of the tangent weights. Therefore, many cohomological statements can be reduced to combinatorial problems. Unfortunately, they are usually complicated. This is the case for the computation of the map $\E_F$.
\\ 

Our idea is to simplify the problem using a reduction to a well-chosen subtorus $\TT_1 \subset \TT$. For a general $\TT_1$, the fixed point $\Hilb{n}^{\TT_1}$ is finite, yet the description of the tangent weights for $\TT_1$ is no easier than for $\TT$.  On the other hand, the combinatorics of the tangent weights visibly simplifies when the fixed point set $\Hilb{n}^{\TT_1}$ is bigger than $\Hilb{n}^{\TT}$. Yet, in such case one needs to consider the geometry of the fixed point set which may be complicated, see \cite{Evain}.
We choose $\Tq \subset \TT$ to be a coordinate subtorus. The fixed point set $\Hilb{n}^{\Tq}$ is much bigger than $\Hilb{n}^{\TT}$, therefore the tangent weights become much easier.  Moreover, $\Hilb{n}^{\Tq}$ is a disjoint union of affine spaces, thus it has trivial geometry (from the cohomological point of view). As expected, after restriction to the subtorus $\Tq$, problem $\ref{intro1}$ simplifies.
\\ \\
Our proof is purely combinatorial but it is based on the geometric intuition presented above. Let $q$ and $t$ be the coordinate characters of $\TT$. The restriction to $\Tq$ corresponds to the substitution $t=1$ in the combinatorial formulas. After this restriction
the adjoint operator $\bigoplus_{n\in\NN}\E^*_{\psi^kB_n}$
satisfies the Leibniz rule, which allows for an effective calculation.
\\ \\
In the last section, we study problem $\ref{intro2}$. We consider the  multiplication \hbox{$\odot:\Lambda^n\times\Lambda^n\to \Lambda^n$} induced by the tensor product in the K-theory ring $\KTh(\Hilb{n})$. This product is determined by its structure constants $c^\mu_{\lambda_1,\lambda_2}\in \QQ$ in the power sum basis, i.e.
$$p_{\lambda_1}\odot p_{\lambda_2}= \sum_{\mu\vdash n} c^\mu_{\lambda_1,\lambda_2} p_\mu\,.$$
In Theorem \ref {tw:main2}, we obtain a formula for the constants $c^\mu_{\lambda_1,\lambda_2}$ as the coefficients of a certain power series. Our method of proof is again the study of the action of the one-dimensional torus $\Tq$.
\\ \\
The structure of this paper is as follows. Sections \ref{s:K}--\ref{s:McKay} provide the necessary background. In Section \ref{s:McKay} we state equivalent formulations of \cite[Conjecture 8.1]{Boissiere} in Theorem \ref{tw:main1} and Proposition \ref{pro:main1}. Section \ref{s:t=1} contains the proof of the main theorem, Theorem \ref{tw:main1}. In Section \ref{s:mult} we study the product $\odot$ on $\Lambda^n$. In Appendix \ref{s:appendix} we include a proof of the equivariant K-theoretic McKay correspondence. It is a well-known result, but we were unable to find a reference.
\\
\paragraph{\bf Acknowledgements:} JK is supported by NCN grant SONATINA 2023/48/C/ST1/00002.
Both authors are grateful to Joachim Jelisiejew, Bruno Stonek and Andrzej Weber for their helpful comments.
JK wants to thank J\"org Sch\"urmann for introducing him to the papers of Boissière.
We would like to thank the anonymous referee for constructive remarks.

\section{K-theory} \label{s:K}
 The standard reference for equivariant K-theory is \cite{CG}. Let $\TT$ be an algebraic torus and $X$ a $\TT$-variety. The equivariant K-theory $\KTh_\TT(X)$ is the $\KTh$-group of the additive category of $\TT$-equivariant vector bundles on $X$. If $X$ is smooth, it is isomorphic to the $\KTh$-group of the category of equivariant coherent sheaves on $X$. All our varieties are quasiprojective. Moreover, we consider only K-theory with rational coefficients, i.e.
	$$\KTh(X):=\KTh(X,\QQ)=\KTh(X,\ZZ)\otimes\QQ\,,\qquad \KTh_\TT(X):=\KTh_\TT(X,\QQ)=\KTh_\TT(X,\ZZ)\otimes\QQ\,.$$

For an algebraic torus $\TT$, let $S_{\TT}\subset \KTh_{\TT}(\pt)$ be the multiplicative system of all nonzero elements. The localised K-theory of a point is $S^{-1}_{\TT}\KTh_{\TT}(\pt)$. For any $\TT$-variety $X$, the ring $\KTh_{\TT}(X)$ is a $\KTh_{\TT}(\pt)$-algebra. We denote the localised K-theory of $X$ by
$$S^{-1}_{\TT}\KTh_{\TT}(X):=\KTh_{\TT}(X)\otimes_{\KTh_{\TT}(\pt)}S^{-1}_{\TT}\KTh_{\TT}(\pt)\,.$$
Analogously, for a finite group $G$ and a $\TT\times G$-variety $X$ we have
$$S^{-1}_{\TT}\KTh_{\TT\times G}(X):=
\KTh_{\TT\times G}(X)\otimes_{\KTh_{\TT}(\pt)}S^{-1}_{\TT}\KTh_{\TT}(\pt)\,.$$
\begin{tw}[Localisation Theorem {\cite[Theorem 2.1]{Tho}}]
	Let $\TT$ be an algebraic torus and $X$ a $\TT$-variety. The inclusion $i:X^\TT\mono X$ induces an isomorphism of rings
	$$i^*:S_{\TT}^{-1}\KTh_\TT(X) \simeq S_{\TT}^{-1}\KTh_\TT(X^\TT)\,.$$
\end{tw}
\begin{cor} \label{cor:fixed}
	Let $\TT$ be an algebraic torus and $X$ be a $\TT$-variety.
	\begin{enumerate}
		\item Suppose that the fixed point set $X^\TT$ is finite. Then the fundamental classes of the fixed points form a basis of the localised K-theory $S_{\TT}^{-1}\KTh_\TT(X)$ over $S_{\TT}^{-1}\KTh_\TT(\pt)$.
		\item More generally, suppose that the fixed point set is a disjoint union of affine spaces. Then the fundamental classes of the components of the fixed point set form a basis of the localised K-theory $S_{\TT}^{-1}\KTh_\TT(X)$ over $S_{\TT}^{-1}\KTh_\TT(\pt)$. 
	\end{enumerate}	
\end{cor}
For a $\TT$-variety $X$, the K-theory $\KTh_\TT(X)$ is a $\lambda$-ring with $\lambda$-operations induced by taking exterior powers, see \cite{lambda1} for  $\lambda$-rings. The theory of $\lambda$-rings provides us with Adams operations
$$
\psi^m: \KTh_{\TT}(X) \to \KTh_{\TT}(X)\,,
$$
for every $m\in \NN$. They were first introduced in the non-equivariant setting in \cite{Adams}.
Let us recall some basic properties of the Adams operations.
\begin{pro}
	Let $\TT$ be an algebraic torus and $X$ be a $\TT$-variety.
	\begin{enumerate}
		\item  The Adams operation $\psi^m: \KTh_{\TT}(X) \to \KTh_{\TT}(X)$ is a ring homomorphism.
		\item  For an equivariant line bundle $\L \in \operatorname{Pic}^\TT(X)$ we have $$\psi^m(\L)=\L^{\otimes m}\,.$$
		\item The Adams operations are natural transformations with respect to pullbacks, i.e. for any $\TT$-equivariant map $f:Y\to X$ we have
		$$f^*\circ \psi^m =\psi^m \circ f^*\,.$$
	\end{enumerate}	  
\end{pro}
\begin{rem} \label{rem:Adams}
	Suppose that a vector bundle $V$ is a direct sum of line bundles
	$$V=L_1\oplus\dots\oplus L_n\,.$$
	Then its Adams power is of the form
	$$\psi^kV=L^{\otimes k}_1\oplus\dots\oplus L^{\otimes k}_n\,.$$
	In general, a vector bundle cannot be decomposed as a direct sum of line bundles. Still the above formula holds if we use the Chern roots of $V$ instead of the line bundles $L_i$. More precisely, let $V$ be any vector bundle and $\alpha_1,\dots\alpha_k$ its Chern roots. Then
	$$\psi^kV=\alpha_1^{k}+\dots+\alpha^{k}_n \in \KTh_\TT(X)\,.$$
\end{rem}
\section{Combinatorics}
\subsection{Partitions}
A partition is a non-increasing sequence of positive integers, i.e.
$$
\lambda=(\lambda_1,\dots,\lambda_k)\,,\qquad \lambda_1\ge\lambda_2\ge\dots\ge\lambda_k >0\,.
$$
We say that $s(\lambda):=\sum_{i=1}^{k} \lambda_i$ is the sum of the partition $\lambda$ or that $\lambda$ is a partition of $s(\lambda)$, and write $\lambda\vdash s(\lambda)$. We call $l(\lambda):=k$ the length of the partition $\lambda$. A partition can be presented in an alternative way $\lambda=(1^{a_1(\lambda)},2^{a_2(\lambda)},\dots)$ where
$$a_j(\lambda):=|\{i\in \NN|\lambda_i=j\}|\,.$$
In further considerations we use the numbers
\begin{equation} \label{w:4}
	z_{\lambda}=\prod_{i=1}^\infty\left( a_i(\lambda)!\cdot i^{a_i(\lambda)}\right)=\prod_{i=1}^\infty a_i(\lambda)!\cdot \prod_{i=1}^{l(\lambda)}\lambda_i\,.
\end{equation}
\begin{df}
	Let $\lambda$ be a partition of $n$ and $\mu$ a partition of $m$.
	The partition $\lambda\cdot\mu$ of $m+n$ is defined by
	$$a_i(\lambda\cdot\mu)=a_i(\lambda)+a_i(\mu)\,. $$
\end{df}
\begin{ex}
	For $\lambda=(5,2,1)$ and $\mu=(3,2,1)$ we obtain 
	$\lambda\cdot\mu=(5,3,2,2,1,1)$.
\end{ex}
\subsection{Symmetric functions} \label{s:sym}
We recall some properties of symmetric functions. For a full exposition and proofs see \cite{Macdonald}. \\
Let $\Perm_m$ be the permutation group of $m$ elements. Denote by $\QQ[x_1,\dots, x_m]_n^{\Perm_m}$ the vector space of invariant homogenous polynomials of degree $n$. We consider a diagram
$$
\QQ[x_1]_n \leftarrow
\QQ[x_1,x_2]^{\Perm_2}_n \leftarrow
\QQ[x_1,x_2,x_3]^{\Perm_3}_n \leftarrow
\dots \leftarrow
\QQ[x_1,\dots, x_m]_n^{\Perm_m} \leftarrow \dots
$$
where the maps are defined by setting the last coordinate to $0$. The $\QQ$-vector space of symmetric functions of degree $n$, denoted $\Lambda^n$, is the inverse limit of this diagram, i.e.
$$\Lambda^n:=\lim\limits_{\leftarrow} (\QQ[x_1]_n \leftarrow
\QQ[x_1,x_2]^{\Perm_2}_n \leftarrow
\dots \leftarrow
\QQ[x_1,\dots, x_m]_n^{\Perm_m} \leftarrow \dots)
$$
We consider also the ring of symmetric functions $\Lambda=\bigoplus_{n\in \NN} \Lambda^n$. Denote by $\Lambda_q^n,\Lambda_{q,t}^n,\Lambda_q,\Lambda_{q,t}$ the spaces of symmetric functions with additional parameters
$$
\Lambda_q^n=\Lambda^n\otimes_\QQ \QQ(q)\,, \qquad
\Lambda_{q,t}^n=\Lambda^n\otimes_\QQ \QQ(q,t)\,, \qquad
\Lambda_q=\Lambda\otimes_\QQ \QQ(q)\,, \qquad \Lambda_{q,t}=\Lambda\otimes_\QQ \QQ(q,t)\,.
$$
The space $\Lambda^n$ has several standard bases indexed by partitions of $n$, such as the power sum functions $p_{\lambda}$, the symmetric monomial functions $m_\lambda$, the full symmetric functions $h_\lambda$ and the Schur functions $s_\lambda$.
\begin{rem}
	The basses $h_\lambda$ and $p_{\lambda}$ are multiplicative, i.e.
	$$
	h_\lambda=\prod_{i=1}^{l(\lambda)} h_{\lambda_i}\,, \qquad
	p_\lambda=\prod_{i=1}^{l(\lambda)} p_{\lambda_i}\,.
	$$
	This implies that for any partitions $\lambda$ and $\mu$ we have $h_{\lambda\cdot\mu}=h_\lambda\cdot h_\mu$ and $p_{\lambda\cdot\mu}=p_\lambda\cdot p_\mu$. \\
	The bases $m_\lambda$ and $s_\lambda$ are not multiplicative.
\end{rem} 
\begin{df} \label{df:w}
	\begin{itemize}
		\item The involution $\omega:\Lambda^n\to \Lambda^n$ is defined by $\omega (p_\lambda)=(-1)^{l(\lambda)}p_\lambda$. We extend it $\QQ(q,t)-$linearly to an involution $\omega:\Lambda_{q,t}^n\to \Lambda_{q,t}^n$.
		\item There is an inner product $\langle -,-\rangle$ on the space $\Lambda^n$  defined by one of the equivalent conditions:
		$$\langle p_\lambda,p_\mu\rangle=\delta_{\lambda,\mu}z_\lambda\,, \qquad \langle s_\lambda,s_\mu\rangle=\delta_{\lambda,\mu}
		\,, \qquad \langle m_\lambda,h_\mu\rangle=\delta_{\lambda,\mu}
		\,.$$
		We extend it to a $\QQ(q,t)$-bilinear form on $\Lambda^n_{q,t}$.
	\end{itemize}
\end{df}
\begin{rem}
	In the literature there are two conventions concerning the involution $\omega$. One either defines $\omega (p_\lambda)=(-1)^{l(\lambda)}p_\lambda$ (e.g. \cite{Boissiere}), or $\omega (p_\lambda)=(-1)^{n+l(\lambda)}p_\lambda$ (e.g. \cite{Macdonald,Haiman_surv}). As we prove a conjecture originating from \cite{Boissiere}, we use the first convention to be consistent with the notation therein. This convention has a nice interpretation in terms of plethystic substitutions, see Example \ref{ex:plethysm} (1).
\end{rem}
The ring of symmetric functions is a polynomial ring generated by the power sums $\{p_n\}_{n\in \NN}$, i.e.
\begin{equation} \label{w:3}
	\Lambda=\QQ[p_1,p_2,\dots]\,.
\end{equation}
This allows us to define derivations $\frac{\partial}{\partial p_n}:\Lambda\to\Lambda$ in the standard way.
\begin{df}
	For a partition $\mu=(\mu_1,\dots,\mu_l)$, the map $\partial_{\mu}:\Lambda\to\Lambda$ is defined by
	$$\partial_{\mu}=\frac{\partial}{\partial p_{\mu_1}}\circ\frac{\partial}{\partial p_{\mu_2}}\circ\dots\circ \frac{\partial}{\partial p_{\mu_l}}\,.$$
	This map may be extended $\QQ(q,t)$-linearly to a map $\partial_{\mu}:\Lambda_{q,t}\to\Lambda_{q,t}$.
\end{df}
The following widely known fact can be easily checked using the power sum basis.
\begin{pro} \label{pro:adjoint}
	Let $f$ and $g$ be symmetric functions.
	\begin{enumerate}
		\item Let $m$ be a natural number.  Multiplication by $\frac{p_m}{m}$ and differentiation $\frac{\partial}{\partial p_m}$ are adjoint, i.e.
		$$\left\langle \frac{f\cdot p_m}{m},g\right\rangle= \left\langle f,\frac{\partial}{\partial p_m}g\right\rangle\,. $$
		\item Let $\lambda$ be a partition. Multiplication by $\frac{p_\lambda}{\prod_{i=1}^{l(\lambda)}\lambda_i}$ and differentiation $\partial_\lambda$ are adjoint, i.e.
		$$\left\langle \frac{f\cdot p_\lambda}{\prod_{i=1}^{l(\lambda)}\lambda_i},g\right\rangle= \left\langle f,\partial_\lambda g\right\rangle\,. $$
	\end{enumerate}
\end{pro}
\subsection{Plethystic substitution}
The plethystic substitution describes the $\lambda$-ring structure on the ring of symmetric functions, see e.g. \cite{plethysm} for a detailed description. Thanks to the presentation \eqref{w:3} of $\Lambda$, in order to define a $\QQ$-algebra map $\varphi:\Lambda^n\to\Lambda^n$ it is enough to specify the values $\varphi(p_n)$.
\begin{df}
	For an element $A\in \Lambda$, the plethystic substitution $f \to f[A]$ is the $\QQ$-algebra morphism $-[A]:\Lambda\to\Lambda$ such that
	$$p_k[A]=A(x_1^k,x_2^k,\dots)\,.$$
	For $A\in \Lambda_{q,t}$, the plethystic substitution is the $\QQ(q,t)$-algebra morphism $-[A]:\Lambda_{q,t}\to\Lambda_{q,t}$ such that $p_k[A]=A(x_1^k,x_2^k,\dots,q^k,t^k)\,.$
\end{df}
\noindent In plethystic notation we write $X$ for the sum of all variables: $X=x_1+x_2+\dots$
\begin{df} \label{df:partq}
	Let $n$ be a natural number and $\lambda$ a partition. We write
	$$[n]_q:=\frac{q^n-1}{q-1}\in\QQ(q)\,,\qquad \Partq{\lambda}:=\prod_{i=1}^{l({\lambda})} [{\lambda}_i]_q\in\QQ(q)\,.$$
\end{df}
\begin{ex} \label{ex:plethysm}
	Let $f,g\in \Lambda_q$ be symmetric functions.
	\begin{enumerate}
		\item We have $f[X]=f$ and $f[-X]=\omega(f)$.
		\item Let $A(q)\in \QQ(q)$ be a non-zero rational function. The substitution $f\to f[A(q)\cdot X]$ is a $\QQ(q)$-linear isomorphism of $\Lambda^n_q$ with inverse \hbox{$f\to f[\frac{X}{A(q)}]$.}
		\item Let $A(q)\in \QQ(q)$ be a rational function. The substitution $f\to f[A(q)\cdot X]$ is self-adjoint, i.e.
		$$\langle f,g[A(q)\cdot X]\rangle= \langle f[A(q)\cdot X],g\rangle $$
		\item For any partition $\lambda$, we have
		$$\langle p_\lambda,f[X(q-1)]\rangle=\langle p_\lambda[X(q-1)],f\rangle=(q-1)^{l(\lambda)}\Partq{\lambda}\cdot\langle p_\lambda,f\rangle \,.$$
	\end{enumerate}
\end{ex}

\subsection{Macdonald Polynomials}
Macdonald polynomials form an important basis of the space $\Lambda_{q,t}$, see \cite[Chapter 6]{Macdonald}. To a partition $\lambda$ of $n$ we associate a symmetric function  $\tilde{H}_\lambda(q,t) \in \Lambda_{q,t}^{n}$ called the modified Macdonald polynomial, see
{\cite[Proposition 2.1.1]{Haiman}}
for a definition. The functions $\tilde{H}_\lambda$ are in fact polynomials in $q,t$. It was proved in \cite[Theorem 2]{Haiman} that their coefficients in the Schur basis are polynomials with non-negative integer coefficients, i.e. 
$$\tilde{H}_\mu(q,t)= \sum_{\lambda \vdash n}\tilde{K}_{\lambda\mu}(q,t)s_\lambda \in \Lambda_{q,t}^{n}\,, \qquad  \tilde{K}_{\lambda\mu}(q,t) \in \NN[q,t] \,.$$
Several restrictions of these polynomials are known, see e.g. \cite[Section 3.5.3]{Haiman_surv}. For us the restriction to $t=1$ will be important.
\begin{pro}[{\cite[p. 364, Example 7]{Macdonald}, see also
		\cite[Proposition 3.5.8]{Haiman_surv} for this formulation}]
	\label{pro:t=1}
	We have
	$$\tilde{H}_\lambda(q,1)=c_\lambda(q) \cdot \prod_{i=1}^{l(\lambda)} h_{\lambda_i}\left[\frac{X}{1-q}\right] \in \Lambda_q^{n}\,,$$
	where $c_\lambda(q)=\prod_{i=1}^{l(\lambda)}\prod_{j=1}^{\lambda_i} (1-q^j)\in \QQ(q)$.
\end{pro}
\begin{ex}
	For $\lambda=(4,2,2,1)$ we have
	$$
	c_\lambda(q)=
	\prod_{j=1}^{4} (1-q^j)\cdot \prod_{j=1}^{2} (1-q^j) \cdot \prod_{j=1}^{2} (1-q^j)\cdot \prod_{j=1}^{1} (1-q^j)=
	(1-q^4)(1-q^3)(1-q^2)^3(1-q)^4\,.
	$$
\end{ex}
\begin{df}\label{df:v}
	Let $\lambda$ be a partition of $n$. We use the  notation
	$$v_\lambda:= h_{\lambda}\left[\frac{X}{1-q}\right] \in \Lambda_q^{n}\,.$$
\end{df}
\begin{pro}
	\label{pro:v}
	\begin{enumerate}
		\item Let $\lambda$ and $\mu$ be partitions. We have
		$$v_{\lambda\cdot\mu}=v_{\lambda}\cdot v_{\mu}\,.$$
		\item The set $\{v_{\lambda}\}_{\lambda \vdash n}$ is a $\QQ(q)$-basis of $\Lambda^n_q$.
	\end{enumerate}
\end{pro}
\begin{proof}
	The proposition follows from analogous properties of the polynomials $h_\lambda$. We have
	$$
	v_{\lambda\cdot\mu}=
	h_{\lambda\cdot\mu}\left[\frac{X}{1-q}\right]=
	(h_{\lambda}\cdot h_{\mu})\left[\frac{X}{1-q}\right]=
	h_{\lambda}\left[\frac{X}{1-q}\right]\cdot h_{\mu}\left[\frac{X}{1-q}\right]=
	v_\lambda\cdot v_\mu\,.
	$$
	For the second point, note that the set $\{h_{\lambda}\}_{\lambda \vdash n}$ is a basis of $\Lambda^n_q$. The plethystic substitution $f\to f[\frac{X}{1-q}]$ is a $\QQ(q)$-linear automorphism of $\Lambda^n_q$, cf. Example \ref{ex:plethysm} (2).
	Therefore the set $\{v_{\lambda}\}_{\lambda \vdash n}$ is also a basis.
\end{proof}
\subsection{Frobenius characteristic map} \label{s:Frob}
The conjugacy classes of the elements of the permutation group $\Perm_n$ are indexed by the partitions of $n$. Therefore, the representation ring $\KTh_{\Perm_n}(\pt)$ is a $\QQ$-vector space with a basis corresponding to partitions. The space $\Lambda^n$ is a $\QQ$-vector space of the same dimension. There is a standard choice of isomorphism between these spaces called the Frobenius characteristic map $\Fr\colon \KTh_{\Perm_n}(\pt) \to \Lambda^n$, see e.g. \cite[Section 7.18]{Stanley}. It is defined by 
$$\Fr([V])=\frac{1}{n!}\sum_{\sigma\in \Perm_n}\chi_V(\sigma) p_{\lambda(\sigma)}\,,$$
where $\chi_V$ is the character of the representation $V$ and $\lambda(\sigma)$ is the partition induced by the permutation $\sigma$.
\noindent Let $\Tq$ be a one-dimensional torus and $\TT$ be a two-dimensional torus. We have isomorphisms of vector spaces
$$\begin{tikzcd}
	{S^{-1}_{\Tq}\KTh_{\Tq\times\Perm_n}(\pt)} & {\KTh_{\Perm_n}(\pt)\otimes_\QQ S^{-1}_{\Tq}\KTh_{\Tq}(\pt)} & {\KTh_{\Perm_n}(\pt)\otimes_\QQ \QQ(q)} & {\Lambda^n_q}
	\arrow["\simeq", from=1-1, to=1-2]
	\arrow["\simeq", from=1-2, to=1-3]
	\arrow["{\Fr\otimes \id_{\QQ(q)}}", from=1-3, to=1-4]
\end{tikzcd},$$
$$\begin{tikzcd}
	{S^{-1}_{\TT}\KTh_{\TT\times\Perm_n}(\pt)} & {\KTh_{\Perm_n}(\pt)\otimes_\QQ S^{-1}_{\TT}\KTh_{\TT}(\pt)} & {\KTh_{\Perm_n}(\pt)\otimes_\QQ \QQ(q,t)} & {\Lambda^n_{q,t}}
	\arrow["\simeq", from=1-1, to=1-2]
	\arrow["\simeq", from=1-2, to=1-3]
	\arrow["{\Fr\otimes \id_{\QQ(q,t)}}", from=1-3, to=1-4]
\end{tikzcd}.$$
We call both these maps the Frobenius map and denote them by $\Fr$. We use this notation only in Definition \ref{df:McKay} and Proposition \ref{pro:sn}.

\section{Hilbert Scheme} \label{s:Hilb}
Let $\Hilb{n} := \Hilb{n}(\Aa_{\CC}^2)$ be the Hilbert scheme of $n$ points in the complex plane $\Aa_{\CC}^2$. We recall some standard facts about $\Hilb{n}$, for reference see \cite{Nak1}. 
The Hilbert scheme $\Hilb{n}$ parametrises zero-dimensional subschemes of $\Aa_{\CC}^2$ of length $n$. It is a smooth irreducible quasi-projective variety of dimension~$2n$, \cite{Fogarty}. Set-theoretically, $\Hilb{n}$ can be identified with the set of ideals $I \lhd \CC[x,y]$ such that the $\CC$-vector space $\CC[x,y]/I$ has dimension $n$, i.e.,
$$\Hilb{n}:=\{I \lhd \CC[x,y]|\, \dim \CC[x,y]/I=n\}\,.$$
Let $\TT=(\CC^{\ast})^2$ be the algebraic torus. The standard $\TT$--action on $\Aa_{\CC}^2$
induces an action of $\TT$ on the Hilbert scheme $\Hilb{n}$. The fixed points of this action are monomial ideals contained in $\Hilb{n}$. They correspond bijectively to partitions of $n$. To a partition $\lambda=(\lambda_1,\dots,\lambda_l)$ we associate an ideal
$$
\langle x^{\lambda_1}\,, x^{\lambda_2}y\,,\dots\,, x^{\lambda_l}y^{l-1}\,, y^l \rangle \in \Hilb{s(\lambda)}^\TT\,.
$$
We do not distinguish in the notation between the partition $\lambda$ and the associated fixed point.
\noindent Denote by $q$ and $t$ the coordinate characters of the torus $\TT$. We consider the one-dimensional coordinate subtorus $\Tq\subset \TT$ acting on the first coordinate. 
Then
$$\KTh_\TT(\pt) = \QQ[q^{\pm 1},t^{\pm 1}]\,,\qquad \KTh_\Tq(\pt)=\QQ[q^{\pm 1}]\,.$$
The restriction map $\KTh_\TT(\pt)\to \KTh_\Tq(\pt)$ is given by the substitution $t=1$. \\

The fixed point set $\Hilb{n}^\Tq$ is a disjoint union of affine spaces, \cite[Section 1.2]{Evain}.  Corollary \ref{cor:fixed} gives the following description of the localised K-theory of $\Hilb{n}$.
\begin{pro} \label{pro:localised}
	\begin{enumerate}
		\item The fundamental classes of the $\TT$-fixed points form a basis of the localised $\KTh$-theory $S^{-1}_\TT\KTh_\TT(\Hilb{n})$.
		\item The fundamental classes of the components of the fixed point set $\Hilb{n}^\Tq$ form a basis of the localised $\KTh$-theory $S^{-1}_\Tq\KTh_\Tq(\Hilb{n})$.
	\end{enumerate}
\end{pro}

\begin{df}
	Let $F$ be a $\TT$-equivariant vector bundle on the Hilbert scheme $\Hilb{n}$ and $\lambda$ a partition of $n$. We write
	$$ F_\lambda \in \KTh^\TT(\pt)\simeq \QQ[q^{\pm 1},t^{\pm 1}]\,,$$
	for the class of the restriction of $F$ to the fixed point $\lambda$.
	We write
	$$ F^q_\lambda \in \KTh^{\Tq}(\pt)\simeq \QQ[q^{\pm 1}] $$
	for the restriction to $\Tq$-equivariant K-theory, i.e. $F^q_\lambda= (F_\lambda)_{t=1}$.
\end{df}
The Hilbert scheme $\Hilb{n}$ comes with a universal flat family
$$\mathcal{F}_n \subset \Hilb{n} \times \Aa^2_\CC\,,$$
whose fibre over a point $Z\in \Hilb{n}$ is $Z$, treated as a subscheme in $\Aa^2_\CC$. The push-forward of the sheaf of regular functions on $\mathcal{F}_n$ along the projection $\pi: \mathcal{F}_n \to \Hilb{n}$ defines (a sheaf of sections of) a rank $n$ vector bundle on $\Hilb{n}$,
$$B_n =\pi_{\ast} \mathcal{O}_{\mathcal{F}_n} \in \operatorname{Vect}(\Hilb{n})\,.$$
It is called the \emph{tautological vector bundle} on $\Hilb{n}$. We are interested in the Adams powers of the tautological bundle, due to the following instance of Kirwan surjectivity.
\begin{tw}[{\cite[Corollary 1.4]{QuiverKirwan}}] \label{tw:Kirwan}
	The classes of the bundles $\psi^kB_n$ for $k\in \NN$ generate $\KTh(\Hilb{n})$ as a $\QQ$-algebra.
\end{tw}
The tautological bundle $B_n$ admits a natural $\TT$-linearisation. Its fibre over a fixed point $I\in \Hilb{n}^\TT$ is isomorphic to the $n$-dimensional vector space $\CC[x,y]/I$ as a $\TT$-representation. Let $\lambda=(\lambda_1,\dots,\lambda_l)$ be a partition of $n$, then
	$$(B_n)_{\lambda}=\sum^{l(\lambda)}_{i=1} \sum_{j=1}^{\lambda_i} t^{i-1}q^{j-1}\,.$$
	This allows for an explicit computation of the restrictions $(\psi^kB_{n})_{\lambda}$ for any $k\in \NN$, see Remark \ref{rem:Adams}:
	$$
	\big(\psi^kB_{n}\big)_{\lambda}=\psi^k\big((B_{n})_{\lambda}\big)=
	\sum^{l(\lambda)}_{i=1} \sum_{j=1}^{\lambda_i} (t^{i-1}q^{j-1})^k\,.
	$$
	After restriction to the one dimensional torus $\TT_q$ we obtain
	$$
	(B_n)^q_{\lambda}=\sum^{l(\lambda)}_{i=1} \sum_{j=1}^{\lambda_i} q^{j-1}
	\,, \qquad
	(\psi^kB_{n})^q_{\lambda}=\sum^{l(\lambda)}_{i=1} \sum_{j=1}^{\lambda_i} (q^{j-1})^k\,.
	$$
The last formula in the above example implies the following fact.
\begin{cor} \label{cor:Adams}
	Let $\lambda$ be a partition of $n$, $\mu$ be a partition of $m$ and $k\in \NN$ be a natural number. We have
	$$ (\psi^kB_{n+m})^q_{\lambda\cdot\mu}=(\psi^kB_{n})^q_{\lambda}+(\psi^kB_{m})^q_{\mu}\,. $$
\end{cor}
\begin{rem}
	The above fact is false if we restrict to $q=1$ instead of $t=1$. In this case, we have
	$$ (\psi^kB_{n+m})^t_{(\lambda^\top\cdot\mu^\top)^\top}=(\psi^kB_{n})^t_{\lambda}+(\psi^kB_{m})^t_{\mu}\,, $$
	where $\lambda^\top$ denotes the transpose partition.
\end{rem}

\section{McKay correspondence} \label{s:McKay}
\subsection{Additive correspondence}
The permutation group $\Perm_n$ acts on the vector space $\CC^{2n}=(\CC^2)^{n}$ by permutation of the $\CC^2$ factors. The Hilbert-Chow morphism \hbox{$\Hilb{n} \to \CC^{2n}/\Perm_n$} is a resolution of the quotient singularity. In \cite[Definition 3.2.4]{Haiman}, the  isospectral Hilbert scheme $X_n$ is defined as the reduced fiber product
\begin{equation} \label{diagram2}
	\begin{tikzcd}
		{X_n} & {\CC^{2n}} \\
		{\Hilb{n}} & {\CC^{2n}/\Perm_n}
		\arrow["\rho", from=1-1, to=1-2]
		\arrow["p"', from=1-1, to=2-1]
		\arrow[from=2-1, to=2-2]
		\arrow[from=1-2, to=2-2]
	\end{tikzcd}.
\end{equation}
We use the names of the maps from the above diagram. The McKay correspondence relates the derived category of the Hilbert scheme with the $\Perm_n$-equivariant derived category of the affine space $\CC^{2n}$.
\begin{pro}[{\cite[Corollary 1.3]{BKR}}]
	The composition
	$$R\rho_*\circ p^*\colon D(\Hilb{n}) \to D^{\Perm_n}(\CC^{2n})$$
	is an equivalence of categories.
\end{pro}
\noindent As a consequence, we have an isomorphism of $\KTh$-groups.
\begin{cor}
	The map
	$$\rho_*\circ p^*\colon \KTh(\Hilb{n}) \to \KTh_{\Perm_n}(\CC^{2n})$$
	is an isomorphism of $\QQ$-vector spaces.
\end{cor}
\begin{cor}\label{cor:eq}
	The maps
	$$\rho_*\circ p^*\colon S^{-1}_{\TT}\KTh_{\TT}(\Hilb{n}) \to S^{-1}_{\TT}\KTh_{\TT\times\Perm_n}(\CC^{2n})\,,\qquad
	\rho_*\circ p^*\colon S^{-1}_{\Tq}\KTh_{\Tq}(\Hilb{n}) \to S^{-1}_{\Tq}\KTh_{\Tq\times\Perm_n}(\CC^{2n})$$
	are isomorphisms of vector spaces over $\QQ(q,t)$ and $\QQ(q)$, respectively.
\end{cor}
\noindent The torus-equivariant version is an easy consequence of the non-equivariant one. It is widely known, see e.g. \cite[Section 5.4.3]{Haiman_surv}, yet we were unable to find a proof in the literature. For the sake of completeness we present a simple argument in Appendix \ref{s:appendix}.
\begin{df}[cf. {\cite[Section 3.2]{Boissiere}}] \label{df:McKay}
	Let $i\colon\{0\} \mono \CC^{2n}$ be the inclusion of the origin and $\Fr$ the Frobenius characteristic map (see Section \ref{s:Frob}).
	We consider the composition of isomorphisms of $\QQ(q,t)$-vector spaces:
	$$
	\Theta_{q,t}\colon \Lambda_{q,t}^n\xto{\Fr^{-1}}
	 S^{-1}_{\TT}\KTh_{\TT\times\Perm_n}(\pt) \xto{(i^*)^{-1}}
	 S^{-1}_{\TT}\KTh_{\TT\times\Perm_n}(\CC^{2n}) \xto{(q_*\circ p^*)^{-1}}
	 S^{-1}_{\TT}\KTh_{\TT}(\Hilb{n})\,.
	$$
	Analogously, we consider isomorphisms
	$$\Theta: \Lambda^n \xtto{\simeq} \KTh (\Hilb{n}) \,, \qquad \Theta_q: \Lambda_q^n  \xtto{\simeq} S^{-1}_\Tq \KTh^\Tq(\Hilb{n})  \,.$$ 
\end{df}
\begin{pro} \label{pro:sn}
	Schur function $s_n$ corresponds to the trivial bundle over Hilbert scheme in the McKay correspondence, i.e.
	$$\Theta_{q,t}(s_n)=1 \in \KTh^\TT(\Hilb{n})\,.$$
\end{pro}
This is a well known fact (e.g. \cite[First dot in Section 8.1]{Boissiere}, \cite[Discussion after Proposition 5.4.6]{Haiman_surv}), yet we were unable to find a direct reference. Below we present a sketch of a proof.
\begin{proof}[Proof]
	Let $\rho_{triv}$ be the trivial representation of the permutation group $\Perm_n$. By definition $s_n$ corresponds to $\rho_{triv}$ in the Frobenius map, i.e. $\Fr^{-1}(s_n)=\rho_{triv}$. \\
	The sheaf $P=p_*\O_{X_n}$ is a  $\Perm_n$-equivariant vector bundle of rank $n!$, see \cite[Theorem 1]{Haiman} and \cite[Theorem 5.2.1]{Haiman2}.
	Representation $\rho_{triv}$ is self-dual, therefore \cite[discussion after Formula 5.3]{ItoNakajima} implies that
	$$
	(q_*\circ p^*)^{-1}(\rho_{triv})=(q_*\circ p^*)^{-1}(\rho^*_{triv})=
	\operatorname{Inv}(P)\,,
	$$
	where $\operatorname{Inv}$ denotes the $\Perm_n$-invariant subbundle. Every fiber of $P$ is isomorphic to regular representation of $\Perm_n$, thus $\operatorname{Inv}(P)$ is a line bundle. The standard map
	$\O_{\Hilb{n}}\to p_*\O_{X_n}$ induces a nowhere vanishing invariant section of $P$. Therefore the bundle $\operatorname{Inv}(P)$ is trivial.
\end{proof}
\subsection{Multiplication by the class of a vector bundle} \label{s:McKay2}
 In this section we recall the multiplicative properties of the McKay correspondence obtained in \cite{Boissiere}. We describe the operation on the  symmetric functions induced by the multiplication by the class of a vector bundle in the $\KTh$-theory of $\Hilb{n}$.
\begin{df}
	Let $F$ be a $\TT$-equivariant vector bundle on the Hilbert scheme $\Hilb{n}$. We consider a $\QQ(q,t)$-linear map $\E^{q,t}_F:\Lambda^n_{q,t} \to \Lambda^n_{q,t}$ defined by
	$\Theta_{q,t}(\E^{q,t}_F(x))=\Theta_{q,t}(x)\otimes F \,,$
	i.e. the following diagram commutes:
$$
	\begin{tikzcd}
		{\Lambda^n_{q,t}} & {S^{-1}_\TT\KTh_\TT(\Hilb{n})} \\
		{\Lambda^n_{q,t}} & {S^{-1}_\TT\KTh_\TT(\Hilb{n})}
		\arrow["{\E^{q,t}_F}"', from=1-1, to=2-1]
		\arrow["{\Theta_{q,t}}", from=1-1, to=1-2]
		\arrow["{-\otimes F}", from=1-2, to=2-2]
		\arrow["{\Theta_{q,t}}", from=2-1, to=2-2]
	\end{tikzcd}.
$$
Analogous definitions may be given when $F$ is a vector bundle or a $\Tq$-equivariant vector bundle.  We consider the $\QQ$-linear map $\E_F:\Lambda^n \to \Lambda^n$, or the $\QQ(q)$-linear map $\E^{q}_F:\Lambda^n_{q} \to \Lambda^n_{q}$ defined by
$$\Theta_{q}(\E^{q}_F(x))=\Theta_{q}(x)\otimes F\,, \qquad \Theta(\E_F(x))=\Theta(x)\otimes F \,.$$
\end{df}
As a straightforward consequence of the above definition we obtain that the maps $\E_F^{q}$ and $\E_F$ are restrictions of $\E_F^{q,t}$ to $\Lambda^n_{q}$ and $\Lambda^n$, respectively.
\begin{pro} \label{pro:restriction1}
	Let $F$ be a $\TT$-equivariant vector bundle on the Hilbert scheme $\Hilb{n}$ and $f\in \Lambda^n_{q,t}$ be a symmetric function. Suppose that the restriction $f_{|t=1} \in \Lambda^n_{q}$ (or $f_{|t=1,q=1}\in \Lambda^n$) is defined. Then $\E_F^{q,t}(f)_{|t=1}$ (or $\E_F^{q,t}(f)_{|t=1,q=1}$) is also defined. Moreover,
	$$\E^q_F(f_{|t=1})=\E_F^{q,t}(f)_{|t=1}\in \Lambda^n_{q}\,, \qquad \E_F(f_{|t=1,q=1})=\E_F^{q,t}(y)_{|t=1,q=1}\in \Lambda^n\,.$$
\end{pro}

\begin{df}
	Let $F$ be a $\TT$-equivariant vector bundle on the Hilbert scheme $\Hilb{n}$. Let $\nabla^{q,t}_F:\Lambda^n_{q,t} \to \Lambda^n_{q,t}$ be the $\QQ(q,t)$-linear map defined by
	$$ \nabla^{q,t}_F\tilde{H}_\lambda=F_\lambda \tilde{H}_\lambda \, \textrm{ for all } \lambda\vdash n.$$
\end{df}
\begin{rem}
	The operators $\nabla^{q,t}_F$ were studied in a  purely combinatorial way in \cite{nabla1}.
\end{rem}

\begin{tw}[{\cite[Theorem 4.1 and its proof]{Boissiere}}] \label{tw:B1}
	Let $F$ be a $\TT$-equivariant vector bundle over $\Hilb{n}$. We have equality of endomorphisms of $\Lambda_{q,t}^n$
	\begin{align*}
		\E^{q,t}_F=\omega\circ(\nabla^{q,t}_F)^*\circ \omega \,,
	\end{align*}
	where the dual is taken with respect to the standard scalar product.
\end{tw}
\begin{rem}
	There is a clear intuition behind the above theorem. The fundamental classes of the fixed points basis $[\lambda]$ form a basis of K-theory, cf. Proposition \ref{pro:localised} (1). It is an eigenbasis of the map $-\otimes F$. Therefore
	$\Theta^{-1}([\lambda])$ is an eigenbasis of the operator $\E_F^{q,t}$.
	The essence of the theorem is a description of the element $\Theta^{-1}([\lambda])$ in terms of Macdonald polynomials.
\end{rem}
\noindent Theorem \ref{tw:B1} and Proposition \ref{pro:restriction1} imply that the map $\nabla_F^{q,t}$ has well-defined restrictions to \hbox{$\Lambda^n_q$ and $\Lambda^n$.}

\begin{df} \label{df:restriction}
	Let $F$ be a $\TT$-equivariant vector bundle over $\Hilb{n}$. We define operators $\nabla_F^{q}:\Lambda^n_q\to \Lambda_q^n$ and $\nabla_F:\Lambda^n\to \Lambda^n$ such that
	$$
	\nabla^{q}_F:=(\omega\E^{q}_F\omega)^*\,, \qquad \nabla_F:=(\omega\E_F\omega)^*\,.
	$$
\end{df}

\begin{cor}
	Let $F$ be a $\TT$-equivariant vector bundle over $\Hilb{n}$ and $x\in \Lambda^n_{q,t}$.
	\begin{enumerate}
			\item Suppose that $x_{|t=1} \in \Lambda^n_{q}$ is defined. Then $\nabla_F^{q,t}(x)_{|t=1}$ is also defined. Moreover,
			$$\nabla_F^{q}(x_{|t=1})=\nabla_F^{q,t}(x)_{|t=1} \in \Lambda^n_{q},.$$
			\item Suppose that $x_{|t=1,q=1} \in \Lambda^n$ is defined. Then $\nabla_F^{q,t}(x)_{|t=1,q=1}$ is also defined. Moreover,
			$$\nabla_F(x_{|t=1,q=1})=\nabla_F^{q,t}(x)_{|t=1,q=1} \in \Lambda^n\,.$$
		\end{enumerate}
\end{cor}
\begin{proof}
	Theorem \ref{tw:B1} and Definition \ref{df:restriction} imply that 
	$$\nabla^{q,t}_F(x)=(\omega\E^{q,t}_F\omega)^*(x)=\sum_{\lambda\vdash n}
	\frac{\langle x, \omega\E^{q,t}_F\omega(p_\lambda) \rangle}{z_\lambda} p_\lambda\,,
	\qquad
	\nabla^{q}_F(x_{|t=1})
	=\sum_{\lambda\vdash n}
	\frac{\langle x_{|t=1}, \omega\E^{q}_F\omega(p_\lambda) \rangle}{z_\lambda} p_\lambda\,.
	$$
	The restriction $x_{t=1}$ is well defined, thus for any partition $\lambda\vdash n$ we have
	$$
	\langle x, \omega\E^{q,t}_F\omega(p_\lambda) \rangle_{|t=1}=
	\langle x_{|t=1}, (\omega\E^{q,t}_F\omega(p_\lambda))_{|t=1} \rangle=
	\langle x_{|t=1}, \omega\E^{q}_F\omega(p_\lambda) \rangle\,.
	$$
	We used Proposition \ref{pro:restriction1} for $f=\omega(p_\lambda)$. This proves the first point.
	\\
	The proof of the second point is analogous.
\end{proof}

\noindent The $\KTh$-theory of the Hilbert scheme is generated as a $\QQ$-algebra by the Adams powers of the tautological bundle, cf. Theorem \ref{tw:Kirwan}. Therefore in order to describe the map $\E_F$ for every $F$ it is enough to describe it for $F=\psi^kB_n\,, k\in \NN$.
\begin{df}
	We write $\nabla_k^{q,t}$ (or $\nabla_k^{q}$, $\nabla_k$) for $\nabla^{q,t}_{\psi^kB}$ (or $\nabla_{\psi^kB}^{q}$, $\nabla_{\psi^kB}$ respectively). In this definition we intentionally omit the number of points in the Hilbert scheme.
\end{df}

The operator $\nabla_1^{q,t}$ can be expressed using   plethystic substitution. The Macdonald polynomial $\tilde{H}_\lambda(q,t)$ is an eigenvector of the operator $D:\Lambda^n_{q,t} \to \Lambda^n_{q,t}$ defined by
\begin{align} \label{w:5}
	D(f)=\operatorname{Coeff} \left(z^0; f\left[X+\frac{(1-q)(1-t)}{z}\right]\cdot \exp\left(\sum_{k\ge 1}\frac{-z^kp_k}{k} \right)\right) 
\end{align}
with eigenvalue $1-(1-q)(1-t)B_{\lambda}$, cf. \cite[Corollary 3.5.5]{Haiman_surv}. Therefore, \hbox{$\nabla_1^{q,t}=\frac{1-D}{(1-q)(1-t)}$.}
\subsection{The main result}
We are now ready to state the main theorem, which gives a closed formula for the multiplication induced by the tensoring with the $k$'th Adams power of the tautological bundle.

\begin{tw}[cf. {\cite[Conjecture 8.1]{Boissiere}}]
	\label{tw:main1}
	Let $k\in \NN$ be a natural number. We have
	$$
	\E_{\psi^kB}=\sum_{\lambda}\frac{k^{l(\lambda)-1}s(\lambda)}{\prod_{i\in\NN} a_i(\lambda)!} p_{s(\lambda)}\partial_{\lambda}\,.
	$$
\end{tw}
\noindent For $k=1$ the above theorem is \cite[Theorem 5.1]{Boissiere}.
Using  $\omega\E_{F}\omega=\nabla_F^*$ (Definition \ref{df:restriction}) and $\omega(p_\lambda)=(-1)^{l(\lambda)}p_\lambda$ (Definition \ref{df:w}), the above theorem can be reformulated in terms of the operator $\nabla_k^*$, as follows.  
\begin{pro} \label{pro:main1}
	Let $k\in \NN$ be a natural number. We have
	\begin{align} \label{w:8}
		\nabla_k^*=\sum_{\lambda}\frac{(-k)^{l(\lambda)-1}s(\lambda)}{\prod_{i\in\NN} a_i(\lambda)!} p_{s(\lambda)}\partial_{\lambda}\,.
	\end{align}
\end{pro}
The next section is devoted to the proof of the above result.
\begin{rem}
	Our proof is purely combinatorial, yet it has a geometric motivation. It is based on the substitution $t=1$, which corresponds to the restriction to $\Tq$-equivariant $\KTh$-theory. The fixed point set $\Hilb{n}^\Tq$ is still relatively simple. Proposition \ref{pro:localised} (2) provides an eigenbasis of the operator $-\otimes F$ acting on $S^{-1}_\Tq\KTh_\Tq(\Hilb{n})$. On the other hand, the fixed point set $\Hilb{n}^\Tq$ is much bigger than $\Hilb{n}^\TT$. This implies that the combinatorics of the tangent weights simplifies after the restriction to $\Tq$. 
	Therefore, the formulas for $\E^q_{F}$ and $\nabla_F^q$ should also simplify.  
\end{rem}
\begin{rem}
	The main step of our proof is showing that the maps $\nabla_k$ and $\nabla_k^q$ satisfy the Leibniz rule in Proposition \ref{pro:Leibniz}:
	$$\nabla_k^{q}(fg)=\nabla_k^{q}(f)\cdot g+f\cdot\nabla_k^{q}(g)\,. $$
	For the non-equivariant map $\nabla_k$ this can be predicted from the conjectured Formula \eqref{w:8}. Multiplication by $\frac{p_m}{m}$ and differentiation $\frac{\partial}{\partial p_m}$ are adjoint. Therefore, we expect that every summand of $\nabla_k$ contains a single differentiation.
\end{rem}
\begin{rem}
	The $\TT$-equivariant map $\nabla_k^{q,t}$ does not satisfy the Leibniz rule even for $p_1^2$.
	This confirms that restriction to the subtorus $\Tq$ simplifies the problem, not only the proofs.
\end{rem}
\begin{rem}
	In \cite{Boissiere}, Theorem \ref{tw:main1} for $k=1$ is stated in the equivalent version
	$$
	\E_B=\operatorname{Coeff}\left(t^0;\sum_{r\ge 1}rp_rt^r\cdot \exp\left(\sum_{r\ge 1}\frac{\partial}{\partial p_r}t^{-r}\right)\right)\,.
	$$
	It is possible to write the result for an arbitrary $k$ in this form
	$$
	\E_{\psi^kB}=k^{-1}\cdot\operatorname{Coeff}\left(t^0;\sum_{r\ge 1}rp_rt^r\cdot \exp\left(k\cdot\sum_{r\ge 1}\frac{\partial}{\partial p_r}t^{-r}\right)\right)\,.
	$$
\end{rem}

\begin{rem}
	Contrary to the cohomological case, the formula for $\E_{\psi^kB}$ does not follow directly from the formula for $\E_{B}$, as the K-theory is not graded. It is tempting to use the Chern character to transfer the results from cohomology to K-theory, yet the relation between the Chern character and the McKay correspondence is surprisingly complicated, \cite[Theorem 1.1]{Boissiere2}.
\end{rem}

\section{Restriction to the one-dimensional torus} \label{s:t=1}
\noindent In this section we study the restriction of the operators $\nabla^{q,t}_F$ to $t=1$. It turns out that the restricted operator is much simpler than the original one. For $F=\psi^kB$ it satisfies the Leibniz rule (Proposition \ref{pro:Leibniz}), which allows us to obtain a closed formula for the operators $\nabla_k$ (Proposition \ref{pro:nabla}). We use Proposition \ref{pro:adjoint} to compute the dual $\nabla^*_k$. Recall that we defined $v_\lambda = h_{\lambda}\left[\frac{X}{1-q}\right] $ (Definition \ref{df:v}).
\begin{pro}
	\label{pro:eigen}
	Let $\lambda$ be a partition. The element $v_\lambda$ is an eigenvector of $\nabla^q_F$ with eigenvalue $F^q_\lambda$.
\end{pro}
\begin{proof}
	Let $c_\lambda(q) \in \QQ(q)$ be the constant from Proposition \ref{pro:t=1}. Then
	\begin{multline*}
		\nabla^q_F(c_\lambda(q)\cdot v_\lambda)=
		\nabla^q_F(\tilde{H}_\lambda(q,1))=
		\nabla^{q,t}_F(\tilde{H}_\lambda(q,t))_{|t=1}= \\ =
		(F_\lambda\cdot\tilde{H}_\lambda(q,t))_{|t=1}=
		F^q_\lambda\cdot \tilde{H}_\lambda(q,1)=
		F^q_\lambda\cdot c_\lambda(q)\cdot v_\lambda\,.
	\end{multline*}
	The first and the last equality follow from Proposition \ref{pro:t=1}. \\
	Thus, $c_\lambda(q)\cdot v_\lambda$ is an eigenvector of $\nabla^q_F$ with eigenvalue  $F^q_\lambda$. It follows that $v_\lambda$ is also an eigenvector with the same eigenvalue.
\end{proof}
\noindent The multiplicative behaviour of the basis $v_\lambda$ implies that the operators $\nabla_k^q$ satisfy the Leibniz rule.
\begin{pro} \label{pro:Leibniz}
	Fix a positive integer $k$. The operator $\nabla_k^{q}$ satisfies the Leibniz rule with regard to the multiplication in $\Lambda_q$, i.e. for $f,g\in \Lambda_q$, we have
	$$\nabla_k^{q}(fg)=\nabla_k^{q}(f)\cdot g+f\cdot\nabla_k^{q}(g)\,. $$
\end{pro}
\begin{proof}
	Let $F=\psi^kB$. It is enough to check the equation on the basis elements $v_\lambda$. We have 
	\begin{align*}
		\nabla_k^{q}(v_\lambda\cdot v_\mu)=
		\nabla_k^{q}(v_{\lambda\cdot\mu})=
		F^q_{\lambda\cdot\mu}v_{\lambda\cdot\mu}=&
		(F^q_{\lambda}+F^q_{\mu})\cdot v_{\lambda,\mu} \\
		=&F^q_{\lambda}v_\lambda\cdot v_\mu+v_\lambda\cdot F^q_{\mu}v_\mu \\
		=&\nabla_k^{q}(v_\lambda)\cdot v_\mu+ v_\lambda\cdot \nabla_k^{q}(v_\mu)\,.
	\end{align*}
	The first and the fourth equalities follow from Proposition \ref{pro:v}, the second and the fifth from Proposition \ref{pro:eigen} and the third from Corollary \ref{cor:Adams}.
\end{proof}
\begin{cor} \label{cor:split}
	Let $\lambda=(\lambda_1,\dots,\lambda_l)$ be a partition and $k\in \NN$ a natural number. 
	Denote by $\widehat{\lambda}_i$ the partition $\lambda$ without the $i$-th part. We have
		$$\nabla^q_k(p_\lambda)=\sum_{i=1}^{l} \nabla^q_k(p_{\lambda_i})p_{\widehat{\lambda}_i}\,. $$
\end{cor}
The above corollary implies that in order to provide a formula for the operator $\nabla^q_k$ it is enough to compute  $\nabla^q_k(p_n)$ for every $n\in \NN$. We start with the case $k=1$.
\begin{pro}\label{pro:k=1}
	We have
	$$\nabla^q_1(p_n)=\sum_{\lambda\vdash n}\znak{\lambda}\cdot \frac{1-q^n}{1-q}\cdot\frac{n}{z_\lambda}\cdot p_{\lambda}\,,\qquad \nabla_1(p_n)=\sum_{\lambda\vdash n}\znak{\lambda} \cdot \frac{n^2}{z_\lambda}\cdot p_{\lambda}\,.$$
\end{pro}
\begin{proof}
Let $D$ be the operator defined by Equation \eqref{w:5}.  By definition
\begin{align*}
	D(p_n)&=
	\operatorname{Coeff} \left(z^0; \left(p_n+\frac{(1-q^n)(1-t^n)}{z^n}\right)\cdot \exp\left(\sum_{k\ge 1}\frac{-z^kp_k}{k} \right)\right) \\
	&=p_n\operatorname{Coeff} \left(z^0;\exp\left(\sum_{k\ge 1}\frac{-z^kp_k}{k} \right)\right)+ (1-q^n)(1-t^n)\operatorname{Coeff} \left(z^n;\exp\left(\sum_{k\ge 1}\frac{-z^kp_k}{k} \right)\right) \\
	&=p_n + (1-q^n)(1-t^n)\cdot\sum_{\lambda\vdash n}\frac{(-1)^{l(\lambda)}}{z_\lambda}p_\lambda
\end{align*}
We have $\nabla_1^{q,t}=\frac{1-D}{(1-q)(1-t)}$, therefore
$$\nabla_1^{q,t}(p_n)= \frac{(1-q^n)(1-t^n)}{(1-q)(1-t)}\cdot\sum_{\lambda\vdash n} \frac{\znak{\lambda}}{z_\lambda} p_\lambda \,.$$
To obtain the result, we compute the limit at $t\to 1$ and then at $q\to 1$.
\end{proof}
\noindent It turns out that the coefficients of $\nabla^q_1(p_n)$ determine the operators $\nabla^q_k$ for an arbitrary $k\in\NN$, as the next two propositions show.
\begin{pro} \label{pro:poly}
	There exist polynomials $C_\lambda(q)\in\QQ[q]$ such that for every $k\in \NN$ we have
	$$\nabla^q_k(p_n)=\sum_{\lambda\vdash n}\frac{(1-q^n)\cdot C_\lambda(q^k)}{\prod_{i=1}^{l(\lambda)}(1-q^{\lambda_i})}p_\lambda \in \Lambda^q_n\,.  $$
\end{pro}
\begin{proof}
	There exist rational numbers $a_{\mu}, b_{\mu\lambda}$ such that
	$$p_n=\sum_{\mu\vdash n}a_{\mu}h_{\mu}\,, \qquad h_{\mu}=\sum_{\lambda\vdash n}b_{\mu\lambda}p_{\lambda}\,. $$
	Applying the plethystic substitution $f\to f[\frac{X}{1-q}]$, we get
	$$\frac{p_n}{1-q^n}=\sum_{\mu\vdash n}a_{\mu}v_{\mu}\,.$$
	It follows that
	$$\nabla_k(p_n)=
	(1-q^n)\sum_{\mu\vdash n}a_{\mu}\nabla_k(v_{\mu})=
	(1-q^n)\sum_{\mu\vdash n}a_{\mu}B^q_{\mu}(q^k) v_{\mu}=
	(1-q^n)\sum_{\mu\vdash n}\sum_{\lambda\vdash n}\frac{a_{\mu}b_{\mu\lambda}B^q_{\mu}(q^k)}{\prod_{i=1}^{l(\lambda)}(1-q^{\lambda_i})} p_\lambda
	\,.
	$$
	Here we used $(\psi^kB)^q_{\mu}(q)=B^q_{\mu}(q^k)$. To obtain the result we set
	\[C_\lambda(q)=\sum_{\mu\vdash n}a_{\mu}b_{\mu\lambda}B^q_{\mu}(q)\,.  \qedhere \] 
\end{proof}
\begin{cor} \label{cor:C}
The polynomial $C_\lambda(q)$ is given by the formula
	$$C_\lambda(q)
	=\znak{\lambda}\cdot \frac{\prod_{i=1}^{l(\lambda)}(1-q^{\lambda_i})}{(1-q)}\cdot \frac{n}{z_\lambda}\,.$$
\end{cor}
\begin{proof}
	We compute the coefficient of $p_{\lambda}$ in $\nabla^q_{1}(p_n)$ using Propositions \ref{pro:k=1} and \ref{pro:poly}. Comparing the results we obtain
	$$	
	\frac{(1-q^n)\cdot C_\lambda(q)}{\prod_{i=1}^{l(\lambda)}(1-q^{\lambda_i})}
	=\znak{\lambda}\cdot\frac{1-q^n}{(1-q)}\cdot \frac{n}{z_\lambda}\,.
	$$
	This implies the statement. 
\end{proof}
\begin{cor}\label{cor:pn}
	We have
	$$\nabla^q_k(p_n)=
	\sum_{\lambda\vdash n}\znak{\lambda}\cdot\frac{1-q^{n}}{1-q^{k}}\cdot \left(\prod_{i=1}^{l(\lambda)}\frac{1-q^{k\lambda_i}}{1-q^{\lambda_i}}\right)\cdot \frac{n}{z_\lambda}\cdot p_{\lambda}\in \Lambda^q_n\,.$$
	Moreover,
	$$\nabla_k(p_n)=\sum_{\lambda\vdash n}\znak{\lambda}\cdot \frac{n^2\cdot k^{l(\lambda)-1}}{z_\lambda}\cdot p_{\lambda} \in \Lambda_n\,.$$
\end{cor}
\begin{proof}
	The first part is a direct consequence of Corollary \ref{cor:C} and Proposition \ref{pro:poly}. The~second is obtained by computing the limit at $q\to1$.
\end{proof}
\begin{pro}\label{pro:nabla}
	We have
	\begin{align} \label{w:9}
		\nabla_k=\sum_{\lambda}\frac{(-k)^{l(\lambda)-1}s(\lambda)^2}{z_\lambda} \cdot p_{\lambda} \frac{\partial}{\partial p_{s(\lambda)}}\,.
	\end{align}
\end{pro}
\begin{proof}
	Both sides of Equation \eqref{w:9} are endomorphisms of $\Lambda$ which satisfy the Leibniz rule.  Therefore, it is enough to check this equality evaluated on the ring generators of $\Lambda$. Corollary \ref{cor:pn} imply that both sides have the same value on the elements $p_n$.
\end{proof}

As a consequence we obtain a proof of Proposition \ref{pro:main1}, which in turn is a reformulation of the main result, Theorem \ref{tw:main1}, i.e. for any $k\in \NN$ we have
$$
\nabla_k^*=\sum_{\lambda}\frac{(-k)^{l(\lambda)-1}s(\lambda)}{\prod_{i\in\NN} a_i(\lambda)!} p_{s(\lambda)}\partial_{\lambda}\,.
$$
\begin{proof}[Proof of Proposition \ref{pro:main1}]
	By Proposition \ref{pro:adjoint} (2), for any partition $\lambda$ we have
	$$
	\left(\frac{s(\lambda)^2}{z_\lambda} \cdot p_{\lambda}\frac{\partial}{\partial p_{s(\lambda)}}\right)^*=
	\frac{s(\lambda)^2}{z_\lambda} \cdot \frac{\prod_{i=1}^{l(\lambda)}\lambda_i}{s(\lambda)}\cdot p_{s(\lambda)}\partial_{\lambda}=
	\frac{s(\lambda)}{\prod_{i\in\NN} a_i(\lambda)!} \cdot  p_{s(\lambda)}\partial_{\lambda}\,,
	$$
	here we treat $p_\lambda$ as a multiplication map $\Lambda \to \Lambda$. Multiply the above expression by $(-k)^{l(\lambda)-1}$ and sum over all partitions $\lambda$. Using Proposition \ref{pro:nabla}, the left hand side is equal to $\nabla_k^*$. The right hand side gives the desired formula.
\end{proof}

\begin{rem}
	An analogous reasoning may be repeated for the operator $\nabla^q_k$. We obtain formulas
	\begin{align*}
	\nabla^q_k&=\sum_{\lambda} \frac{(-1)^{l(\lambda)-1}s(\lambda)}{z_\lambda}
	\cdot \frac{1-q^{s(\lambda)}}{1-q^{k}}\cdot \left(\prod_{i=1}^{l(\lambda)}\frac{1-q^{k\lambda_i}}{1-q^{\lambda_i}}\right)
	\cdot
	p_{\lambda}\frac{\partial}{\partial p_{s(\lambda)}}\,,
	\\
	(\nabla^q_k)^*&=\sum_{\lambda} \frac{(-1)^{l(\lambda)-1}}{\prod_{i\in\NN} a_i(\lambda)!}
	\cdot \frac{1-q^{s(\lambda)}}{1-q^{k}}\cdot \left(\prod_{i=1}^{l(\lambda)}\frac{1-q^{k\lambda_i}}{1-q^{\lambda_i}}\right)
	\cdot p_{s(\lambda)}\partial_{\lambda}\,.
	\end{align*}
\end{rem}

\section{Multiplicative structure} \label{s:mult}
The tensor product in the K-theory of the Hilbert scheme induces a product on the space of symmetric functions. In this section we provide an explicit formula for this product in the power sum basis. 
\begin{df}
	The multiplication $\odot:\Lambda^n\times\Lambda^n\to \Lambda^n$ is defined as the map induced by the tensor product in $\KTh(\Hilb{n})$, i.e.
	$$ f \odot g = \Theta^{-1}(\Theta(f)\otimes\Theta(g))\,.$$ 
	The multiplication $\odot_q:\Lambda_q^n\times\Lambda_q^n\to \Lambda_q^n$ is defined as the map induced by the tensor product in $S_\Tq^{-1}\KTh^\Tq(\Hilb{n})$, i.e.
	$$ f \odot_q g = \Theta_q^{-1}(\Theta_q(f)\otimes\Theta_q(g))\,.$$
\end{df}
 \begin{df}
 	Let $\lambda_1,\lambda_2$ and $\mu$ be partitions of $n$. The structure constants $c^\mu_{\lambda_1,\lambda_2} \in \QQ$ and the equivariant structure constants $c^\mu_{\lambda_1,\lambda_2}(q) \in \QQ(q)$ are defined by 
 	$$
 	p_{\lambda_1}\odot p_{\lambda_2}= \sum_{\mu\vdash n} c^\mu_{\lambda_1,\lambda_2} p_\mu\,,\qquad
 	p_{\lambda_1}\odot_q p_{\lambda_2}= \sum_{\mu\vdash n} c^\mu_{\lambda_1,\lambda_2}(q) p_\mu\,.
 	$$
 \end{df}
The product $\odot$ is fully determined by its structure constants. We give a formula for the constant $c^\mu_{\lambda_1,\lambda_2}$ as a coefficient of a certain power series. First we introduce some notation.
For a power series $W(q)\in \QQ[[q-1]]$ and a natural number $m$, we write
$\operatorname{Coeff}_{q-1}(m;W(q))$
for the coefficient of $W(q)$ at $(q-1)^m$. We use the function
$$W_\lambda(q)= (q-1)^n\left\langle s_n, h_\lambda\left[\frac{X}{q-1}\right]\right\rangle\,.$$
An easy computation (cf. Example \ref{ex:plethysm} (4)) shows that
$$W_\lambda(q)=\sum_{\nu\vdash n} (q-1)^{n-l(\nu)} \frac{\left\langle p_{\nu}, h_\lambda\right\rangle}{z_\nu} \cdot  \Partq{\nu}^{-1}\,.$$
 See Definition \ref{df:partq} for the symbol $\Partq{\nu}$.
 The inner product $\langle p_{(1)^n}, h_\lambda\rangle$ is non-zero for any $\lambda$, thus $W_\lambda(q)$ has a non-zero restriction at $q=1$.

\begin{tw} \label{tw:main2}
	Let $\lambda_1,\lambda_2$ and $\mu$ be partitions of $n$. Then
	$$ c^\mu_{\lambda_1,\lambda_2}=\operatorname{Coeff}_{q-1}
	\left(
	{l(\lambda_1)+l(\lambda_2)-n-l(\mu);
	\frac{\Partq{\mu}}
	{z_\mu\Partq{\lambda_1}\Partq{\lambda_2}}\cdot\sum_{\nu\vdash n} \frac{\langle p_{\lambda_1},h_\nu\rangle\langle p_{\lambda_2},h_\nu\rangle\langle p_\mu,m_\nu\rangle}{W_\nu(q)} 
	}
	\right)
	$$
\end{tw}
Before we start the proof of the above theorem let us note a corollary and two remarks.
\begin{cor} \label{cor:coh}
	Suppose that $l(\lambda_1)+l(\lambda_2)-n-l(\mu)$ is negative, then $c^\mu_{\lambda_1,\lambda_2}=0$.
\end{cor}
\begin{proof}
	The power series in Theorem \ref{tw:main2} has a well-defined restriction at $q=1$.
\end{proof}
\begin{rem} \label{rem:odot_dim}
	The above corollary agrees with \cite[Theorem 1.1]{Boissiere2}. The condition that $l(\lambda_1)+l(\lambda_2)-n-l(\mu)$ is negative is equivalent to
	$$(n-l(\lambda_1))+(n-l(\lambda_2))>n-l(\mu)\,.$$
	By \cite[Theorem 1.1]{Boissiere2}, for every partition $\lambda$ the class $\ch(\Theta(p_\lambda))$ has non-zero graded components in $\coh^{2(n-l(\lambda))}(\Hilb{n})$ and is equal to zero in smaller gradations. Corollary \ref{cor:coh} may be deduced from this theorem and the fact that the cup product in cohomology preserves gradation.
\end{rem}
\begin{rem}
	 For small $n$ the formula from Theorem \ref{tw:main2} can be implemented using computer algebra software. We used SageMath to compare our results with computations of \cite[Section 9]{Boissiere}. They agree with one exception: we obtained that $p_{22}\odot p_{22}=0$ instead of $8p_4$. We believe that this is a misprint in \cite{Boissiere}, the product $p_{22}\odot p_{22}$ has to be zero due to dimensional reasons: Remark \ref{rem:odot_dim}, or \cite[second dot in Section 8.1]{Boissiere}.
\end{rem}
The rest of this section is devoted to the proof of Theorem \ref{tw:main2}. First, we study the $\Tq$-equivariant case. The tensor product in $S^{-1}_\Tq\KTh_{\Tq}(\Hilb{n})$ has an easy presentation in the fixed point basis, cf. Proposition \ref{pro:localised} (2). For a partition $\lambda$, let $[X_\lambda]$ be the class of the corresponding component of the fixed point set $\Hilb{n}^\Tq$. Then
\begin{align} \label{w:6}
	[X_\lambda]\otimes [X_\mu]=\delta_{\lambda,\mu}a_\lambda[X_\lambda] \,,
\end{align}
for some rational functions $a_\lambda \in \QQ(q)$. This observation allows us to describe the product $\odot_q$ in Propositions \ref{pro:odot2} and \ref{pro:odot3}. Next, we study the connection between the products $\odot_q$ and $\odot$ in Proposition \ref{pro:odot1}. To obtain a formula for $\odot$ we need to change the basis. We use the following elementary linear algebra lemma.
\begin{lemma} \label{lem:odot1}
	Let $V$ be a finite-dimensional vector space equipped with a non-degenerate bilinear form $\langle-,-\rangle$. Let $A$ be an endomorphism of $V$. Suppose that $\alpha_1,\dots,\alpha_k$ is an eigenbasis for $A$ with eigenvalues $a_1,\dots,a_k$. Then $\alpha^*_1,\dots,\alpha^*_k$ is an eigenbasis for the adjoint map $A^*$ with the same eigenvalues. 
\end{lemma}
\begin{proof}
	It is enough to show that for any $i,j\in\{1,\dots, k\}$ we have
	$\langle A^*\alpha^*_i,\alpha_j\rangle=\langle a_i\alpha^*_i,\alpha_j\rangle$. This is straightforward, as
	\[
	\langle A^*\alpha^*_i,\alpha_j\rangle=\langle \alpha^*_i,A\alpha_j\rangle=
	a_j \cdot \langle \alpha^*_i,\alpha_j\rangle =a_j\delta_{ij}=a_i\delta_{ij}= \langle a_i\alpha^*_i,\alpha_j\rangle\,.
	\] \qedhere
\end{proof}
\begin{lemma} \label{lem:odot2}
	The set $\{m_\lambda[X(q-1)]\}_{\lambda\vdash n}$ is an eigenbasis of the map $\E^q_{B_n}$.
\end{lemma}
\begin{proof}
	The set $\{h_\lambda[\frac{X}{1-q}]\}_{\lambda\vdash n}$ is an eigenbasis of $\nabla^q_{1}$ by Proposition \ref{pro:eigen}. Therefore,
	$$\left\{\omega\left(h_\lambda\left[\frac{X}{1-q}\right]\right)\right\}_{\lambda\vdash n}=
	\left\{h_\lambda\left[\frac{X}{q-1}\right]\right\}_{\lambda\vdash n}$$
	is an eigenbasis of $\omega\nabla^q_{1}\omega$. We have that $(\E^q_{B})^*=\omega\nabla^q_{1}\omega$, so by Lemma \ref{lem:odot1} it is enough to prove that the bases $\{h_\lambda[\frac{X}{q-1}]\}_{\lambda\vdash n}$ and $\{m_\lambda[X(q-1)]\}_{\lambda\vdash n}$ are mutually dual. This is the case, as
	$$
	\left\langle h_\lambda\left[\frac{X}{q-1}\right], m_\mu[X(q-1)]\right\rangle=
	\left\langle h_\lambda, (m_\mu[X(q-1)])\left[\frac{X}{q-1}\right]\right\rangle=
	\left\langle h_\lambda, m_\mu\right\rangle =\delta_{\lambda\mu}\,.
	$$
	The first equality follows from Example \ref{ex:plethysm} (3),
	the second from Example \ref{ex:plethysm} (2).
\end{proof}

\begin{pro} \label{pro:odot2}
	Let $\lambda$ and $\mu$ be partitions of $n$. We have
	$$
	m_\lambda[X(q-1)]\odot_q m_\mu[X(q-1)]=\delta_{\lambda\mu}b_\lambda m_\mu[X(q-1)]
	$$
	for some rational function $b_\lambda \in \QQ(q)$.
\end{pro}
\begin{proof}
	The fundamental classes $[X_\lambda]\in \KTh^\Tq(\Hilb{n})$ form an eigenbasis of the map \hbox{$-\otimes B_n$.} Therefore, the elements $\Theta^{-1}[X_\lambda]$ form an eigenbasis of $\E^q_{B_n}$. The eigenvalues of $\E^q_{B_n}$ are pairwise different, therefore Lemma \ref{lem:odot2} implies that
	$$\Theta^{-1}[X_\lambda]=d_\lambda\cdot m_\lambda[X(q-1)]$$
	for some nonzero $d_\lambda \in \QQ(q)$. The proposition now follows from Equation \eqref{w:6}.
\end{proof}

\begin{pro} \label{pro:odot3}
	For arbitrary symmetric functions $f,g\in\Lambda_q^n$, we have
	$$
		f\odot_q g
		=
		\sum_{\mu\vdash n}p_\mu\cdot\frac{(q-1)^n}{z_\mu}\cdot\sum_{\nu\vdash n}
		\frac{\langle f,h_\nu[\frac{X}{q-1}]\rangle\langle g,h_\nu[\frac{X}{q-1}]\rangle\langle p_\mu,m_\nu[X(q-1)]\rangle}{W_\nu(q)} 
		\,.
	$$
\end{pro}
\begin{proof}
	Proposition \ref{pro:odot2} implies that
	\begin{align} \label{w:7}
		f\odot_q g=\sum_{\nu\vdash n}
		\left\langle f, h_\nu\left[\frac{X}{q-1}\right]\right\rangle
		\left\langle g, h_\nu\left[\frac{X}{q-1}\right]\right\rangle
		b_\nu \cdot m_\nu[X(q-1)]\,
	\end{align}
	for some $b_\nu \in \QQ(q)$. 
	It is a standard fact that $\Theta_q(s_n)$ is the trivial bundle, see Proposition \ref{pro:sn}. 
	Therefore $s_n$ is the neutral element for $\odot_q$. For $g=s_n$ and $f=m_\lambda[X(q-1)]$, Equation \eqref{w:7} gives
	$$m_\lambda[X(q-1)]=
	\sum_{\nu\vdash n}
	\delta_{\lambda\nu}
	\left\langle s_n, h_\nu\left[\frac{X}{q-1}\right]\right\rangle
	b_\nu \cdot m_\nu[X(q-1)]\,.
	$$
	It follows that $$b_\lambda=\left\langle s_n; h_\lambda\left[\frac{X}{q-1}\right]\right\rangle^{-1}=\frac{(q-1)^n}{W_\lambda(q)}\,.$$
	Moreover, $m_\nu[X(q-1)]=\sum_{\mu\vdash n} \langle p_\mu,m_\nu[X(q-1)]\rangle \frac{p_\mu}{z_\mu}\,. $ After substituting these two equations into \eqref{w:7} we obtain the desired formula.
\end{proof}

\begin{pro} \label{pro:odot1}
	Suppose that the symmetric functions $f, g\in \Lambda_q^n$ have well-defined restrictions at $q=1$. Then $f \odot_q g$ also has a well-defined restriction at $q=1$. Moreover,
	$$ (f \odot_q g)_{|q=1}=f_{|q=1} \odot g_{|q=1}\,.$$
\end{pro}
\begin{proof}
	An analogous property is satisfied by the products in $S_\Tq^{-1}\KTh^\Tq(\Hilb{n})$ and $\KTh(\Hilb{n})$.
\end{proof}

\begin{proof}[Proof of Theorem \ref{tw:main2}]
	Proposition \ref{pro:odot3} and Example \ref{ex:plethysm} (4)
	imply that the formula for the equivariant structure constants is the following:
	\begin{align*}
		c^\mu_{\lambda_1,\lambda_2}(q)
		&=(q-1)^{n+l(\mu)-l(\lambda_1)-l(\lambda_2)}\frac{\Partq{\mu}}
		{z_\mu\Partq{\lambda_1}\Partq{\lambda_2}}\cdot\sum_{\nu\vdash n} \frac{\langle p_{\lambda_1},h_\nu\rangle\langle p_{\lambda_2},h_\nu\rangle\langle p_\mu,m_\nu\rangle}{W_\nu(q)}.
	\end{align*}
	By Proposition \ref{pro:odot3} the above expression has a well-defined restriction at $q=1$. 
\end{proof}

\appendix
\section{Equivariant McKay correspondence} \label{s:appendix}
\begin{df}
	Let $k$ be a field and $V$ be a $k$-vector space of finite dimension. The restriction
	$$\res_{q=1}\colon V\otimes_k k(q) \dashrightarrow V$$
	is defined on a submodule $V\otimes k[q]_\mathfrak{m}$, where $\mathfrak{m}$ is the ideal spanned by $(q-1)$.
\end{df}
\begin{df}
	Let $k$ be a field, $V$ and $W$ be $k$-vector spaces of finite dimension and $\psi:V\to W$ be a $k$-linear map. We say that a $k(q)$-linear map $\psi_q:V\otimes_k k(q)\to W\otimes_k k(q)$ restricts to $\psi$ if the following diagram commutes:
	$$
	\begin{tikzcd}
		{V\otimes_k k(q)} & {W\otimes_k k(q)} \\
		V & W
		\arrow["{\psi_q}", from=1-1, to=1-2]
		\arrow["\psi", from=2-1, to=2-2]
		\arrow["{\res_{q=1}}"', dashed, from=1-1, to=2-1]
		\arrow["{\res_{q=1}}", dashed, from=1-2, to=2-2]
	\end{tikzcd},
$$
in the sense that if for $\vv\in V\otimes_k k(q)$ the restriction $\res_{q=1}\vv$ is defined, then the restriction $\res_{q=1}\psi_q(\vv)$ is also defined and equal to $\psi(\res_{q=1}\vv)$.
\end{df}
\begin{lemma}
	Let $k$ be a field, let $V$ and $W$ be $k$-vector spaces of finite dimension, and let \hbox{$\psi:V\to W$} be a $k$-linear map. Let $\psi_q:V\otimes_k k(q)\to W\otimes_k k(q)$ be a $k(q)$-linear map which restricts to $\psi$. Suppose that $\psi$ is an isomorphism. Then $\psi_q$ is also isomorphism.
\end{lemma}
\begin{proof}
	Let $v_1,\dots,v_k$ be a basis of $V$. Then $v_1\otimes 1,\dots,v_k\otimes 1$ is a $k(q)$-basis of $V\otimes k(q)$. It is enough to prove that the vectors $\psi_q(v_1\otimes 1),\dots,\psi_q(v_k\otimes 1)$ are $k(q)$-linearly independent. Suppose otherwise that there exist rational functions $a_k(q) \in k(q)$ such that
	\begin{align} \label{w:2}
		\sum_{i=1}^{k} a_k(q)\psi_q(v_i\otimes 1)=0\,. 
	\end{align}
	Without loss of generality we can assume that all $a_k(q)$ are non-zero. Let $\val_{q-1}$ be the valuation associated with $(q-1)$ and
	$$
	t=\min\{\val_{q-1}(a_1(q)), \dots, \val_{q-1}(a_k(q))\}\,.
	$$
	 Multiplying Equation \eqref{w:2} by $(q-1)^{-t}$ and applying $\res_{q=1}$ yields a non-trivial $k$-linear relation between elements $\res_{q=1}(\psi_q(v_i\otimes 1))$. The map $\psi_q$ restricts to  $\psi$, therefore $$\res_{q=1}(\psi_q(v_i\otimes 1))=\psi(v_i)\,.$$ This contradicts that $\psi$ is an  isomorphism.
\end{proof}
\begin{proof} [Proof of Corollary \ref{cor:eq}]
	 Let $p$ and $\rho$ denote the maps from the diagram \eqref{diagram2}. Let $\rho_*$, $\rho^\Tq_*$, $\rho^\TT_*$, $p^*$, $p^*_\Tq$, $p^*_\TT$ denote the induced maps on K-theory, $\Tq$-equivariant K-theory and $\TT$-equivariant K-theory, respectively. \\
	 Set $k=\QQ\,, V=\KTh(\Hilb{n})\,, W=\KTh_{\Perm_n}(\pt)$. We have
	 $$V\otimes_k k(q)=S^{-1}_\Tq\KTh_{\Tq}(\Hilb{n})\,, \qquad W\otimes_k k(q)=S^{-1}_\Tq\KTh_{\Tq\times\Perm_n}(\pt)\,. $$
	 The  map $\rho^{\Tq}_*\circ p_{\Tq}^*$ restricts to $\rho_*\circ p^*$ which is an isomorphism. Therefore $\rho^{\Tq}_*\circ p_{\Tq}^*$ is also an isomorphism. \\
	 Set $k=\QQ(q)\,, V=S^{-1}_\Tq\KTh_{\Tq}(\Hilb{n})\,, W=S^{-1}_\Tq\KTh_{\Tq\times\Perm_n}(\pt)$. We have
	 $$V\otimes_k k(t)=S^{-1}_\TT\KTh_{\TT}(\Hilb{n})\,, \qquad W\otimes_k k(t)=S^{-1}_\TT\KTh_{\TT\times\Perm_n}(\pt)\,. $$
	 The  map $\rho^{\TT}_*\circ p_{\TT}^*$ restricts to $\rho^{\Tq}_*\circ p_{\Tq}^*$ which is an isomorphism. Therefore $\rho^{\TT}_*\circ p_{\TT}^*$ is also an isomorphism.
\end{proof}


\end{document}